\documentclass[a4paper,12pt]{amsart}
\usepackage{amsmath}
\usepackage{amssymb}
\usepackage{mathrsfs}
\usepackage{enumerate}
\usepackage{ifthen}
\usepackage{graphicx}
\usepackage[T1]{fontenc} %skandit

\nonstopmode \numberwithin{equation}{section}
\newtheorem{thm}{Theorem}[section]
\newtheorem{cor}{Corollary}[section]

\theoremstyle{definition}

\newcounter{minutes}
\divide\time by 60
\newcounter{hours}
\multiply\time by 60
\addtocounter{minutes}{-\time}
\newcounter {own}
\def\theown {\thesection       .\arabic{own}}

{\qed\bigskip}

\newcounter{alphabet}

\begin{document}
	
	\title{On the pre-Schwarzian norm estimate of special close-to-convex harmonic mappings}

	\author{Sushil Pandit}
	\address{Sushil Pandit,
		Department of Mathematics,
		Government Engineering College Palamu,
		Lesliganj-822118, India, Orcid Id-0000-0002-4129-6456}
	\email{sushilpandit15594@gmail.com}

	
	\def\thefootnote{}
	\footnotetext{ {\tiny File:~\jobname.tex,
			printed: \number\year-\number\month-\number\day,
			\thehours.\ifnum\theminutes<10{0}\fi\theminutes }
	} \makeatletter\def\thefootnote{\@arabic\c@footnote}\makeatother

	\begin{abstract}
		In this article, we consider a class of close-to-convex harmonic mappings with special analytic part in the unit disk $\mathbb{D}=\{z\in\mathbb{C}:|z|<1\}$ and obtain sharp pre-Schwarzian norm estimate. We study the theory of harmonic Bloch mapping for the considered class. Moreover, we discuss some growth and distortion theorems for analytic and co-analytic parts of such harmonic mappings. 
	\end{abstract}

	\thanks{}
	
	\maketitle
	\pagestyle{myheadings}
	\markboth{Sushil Pandit}{On the pre-Schwarzian norm estimate of special close-to-convex harmonic mappings}
	%	\mark{Certain properties of analytic Bloch functions}
	
	%%%%%%%%%%%%%%%%%%%%%%%%%%%%%%%%%%%%%%%%%%%%%%%%%%%%%%%%%%%%%%%%%%%%%%%%%%%%%%%%%%%%%%%%%%%%%%%%%%%%%%%%%%%%%%%%%%%%%%%%%%%%%%%%%%%%%%%%%%%%%%
	\section{Introduction}
The pre-Schwarzian derivative of locally univalent analytic mappings has become a well known technique to find either necessary or sufficient conditions for the global univalence, or to obtain certain geometric conditions on the range of the related functions. Due to such useful properties of the derivative, the technique has been extended to harmonic mappings (see \cite{Kanas-Smet-2014, Hernandez-Martin-2015}). The important question is to put suitable and global definition of the derivative and the same is well settled in \cite{ Hernandez-Martin-2015}. In this article, we consider a class of close-to-convex harmonic mappings with special analytic part and study the pre-Schwarzian norm estimate. Moreover, we discuss some growth and distortion theorems for such mappings.\\

	Let $\mathcal{A'}$ denote the class of all analytic functions $h$ in the unit disk $\mathbb{D}=\{z\in\mathbb{C}:|z|<1\}.$
	Let $\mathcal{A}$ be the class of functions $h\in\mathcal{A'}$ of the form
	\begin{align}\label{P800}
		h(z)=z+\sum\limits_{n=2}^\infty a_n z^n
	\end{align}
	
	with the normalization $h(0)=h'(0)-1=0.$ Let $\mathcal{S}$ be the class of univalent functions $h\in\mathcal{A}$.
		\subsection{Analytic functions convex in one direction:}
	 A function $h\in\mathcal{S}$ is called convex in $\mathbb{D}$ if the image domain $h(\mathbb{D})$ is convex. A domain $E$ is called convex if line segment joining any two points in $E$ lies in $E.$ Let $\mathcal{C}$ denotes the class of all convex functions $h\in\mathcal{S}$. It is well-known that (see \cite{Duren-1983}) a function $h\in\mathcal{C}$ is characterized by
	\begin{align*}
		{\rm Re\,}\left(1+z\frac{h''(z)}{h'(z)}\right)>0
	\end{align*}
	for $z\in\mathbb{D}.$ From the above characterization, we see that the function $h(z)=z/(1-z)$ is univalent and convex in $\mathbb{D}$. In fact, the function $h$ maps $\mathbb{D}$ univalently onto the convex domain ${\rm Re\,}w >-1/2$. Umezawa \cite{Umezawa-1952} studied that, a locally univalent function
	
		\begin{align*}
	h\in\mathcal{F}=\left\{\phi\in\mathcal{A}:\alpha>{\rm Re\,}\left(1+z\frac{h''(z)}{h'(z)}\right)>-\frac{\alpha}{2\alpha-3},~\alpha\geq 3/2\right\}	
	\end{align*}
	 is univalent in $\mathbb{D}$ and maps $\mathbb{D}$  onto a domain which is convex in one direction. A	domain $E$ is called convex in the direction $\alpha\in[0,\pi)$ if every line parallel to the line through $0$ and $e^{i\alpha}$ has a connected or empty intersection with $E$. A function $h\in\mathcal{A}$ is called convex in the direction $\alpha$, if it maps $\mathbb{D}$
	to the domain convex in direction $\alpha.$
%	subclasses of S
%	F¼ f 2S : R 1 þ zf 00ðzÞ
%	f 0
%	ðzÞ
%	[  1
%	2
%	; z 2 D
%	; ð3Þ
%	G¼ f 2S : R 1 þ zf 00ðzÞ
%	f 0
%	ðzÞ
%	\
%	3
%	2
%	; z 2 D
%	; ð4Þ
%	respectively. Observe that, the inequality appeared in (3) is
%	a consequence of Kaplan characterization [2, p. 48, Theorem 2.18], hence functions in F are also close-to-convex
%	in D. 
	\subsection{Harmonic mappings:}	
 A continuous twice differentiable complex valued function $f=u+iv$ is called harmonic in a domain $E$ if both $u$ and $v$ are real harmonic in $E.$ In any simply connected domain $E,$ every harmonic mapping $f$ can be written as $f=h+\overline{g},$ where $h$ and $g$ are analytic in $E.$ This representation is known as canonical representation and $h$ is called analytic part whereas $g$ is called co-analytic part of $f.$ A harmonic mapping $f=h+\overline{g}$ is sense preserving if the Jacobian $J_f=|h'|^2-|g'|^2$ is positive and sense reversing if $J_f$ is negative. Lewy \cite{Lewy-1936} showed that $f$ is locally univalent if $J_f$ is non vanishing. Let $\mathcal{H}$ be class of all sense preserving harmonic functions $f=h+\overline{g}$ in $\mathbb{D}$ of the form
    \begin{align}\label{P-151}
		h(z)=z+\sum\limits_{n=2}^\infty a_nz^n,~g(z)=\sum\limits_{n=1}^\infty b_nz^n.
	\end{align}
	Let $\mathcal{S_H}$ denote the family of functions in $\mathcal{H}$ that are univalent in $\mathbb{D}$. In 1984, Clunie and Sheil-Small \cite{Clunie-Small-1984} studied several geometric properties of functions in $\mathcal{S_H}$ alongwith functions in its  subclasses of starlike functions, convex functions, close-to-convex functions.
	
	\subsection{Bloch Functions:}
	 
	The study of behavior of the quantity $|h'(z)|(1-|z|)$ as $|z|\rightarrow 1$ from the interior of the unit circle (see \cite{Seidel-Walsh-1942}) is one of the main reason of discovery of the concept of Bloch functions. An analytic function $h$  defined in $\mathbb{D}$ is called Bloch if $\beta_h=\sup_{z\in\mathbb{D}}(1-|z|^2)|h'(z)|<\infty.$	The collection $\mathcal{B}$ of analytic Bloch functions in $\mathbb{D}$ form a Banach space with the norm given by
	\begin{align*}
		||h||_{\mathcal{B}}=|h(0)|+\sup_{z\in\mathbb{D}}(1-|z|^2)|h'(z)|.
	\end{align*} 
	For more properties of analytic Bloch functions, we refer to the articles \cite{Pommerenke-1970, Anderson-Clunie-Pommerenke-1974,  Anderson-Shields-1976, Bonk-1991, Bonk-Minda-Yanagihara-1996}. A harmonic mapping $f=h+\overline{g}\in\mathcal{H}$ is called harmonic Bloch if and only if
	\begin{align*}
		\beta_f=\sup_{z\in\mathbb{D}}(1-|z|^2)(|h'(z)|+|g'(z)|)<\infty.
	\end{align*}
	For more information about harmonic Bloch mappings, we refer to \cite{Ali-Pandit-2023, Chen-Gauthier-Hengartner-2000, Colonna-1989}.

	\subsection{Pre-Schwarzian norms:}
	For a locally univalent analytic function $h$, the quantity 
	\begin{align}\label{P-141}
	||P_h|| = \sup_{z \in \mathbb{D}}(1-|z|^2)|P_h(z)|
	\end{align}
	is known as pre-Schwarzian norm of $h.$ Here, $P_h=\frac{h''}{h'}$ is called pre-Schwarzian derivative of $h.$ For a univalent function $h\in\mathcal{A}$, we have the sharp estimate $||P_h||\leq 6$ (see \cite{Kraus-1932}). On the other hand, for a locally univalent function $h\in\mathcal{A}$, if $||P_h||\leq 1$ (see \cite{Becker-1972}, \cite{Becker-Pommerenke-1984}), then the function $h$ is univalent.
%	 In 1976, Yamashita \cite{Yamashita-1976} proved that $||P_h||$ is finite if and only if $h$ is uniformly locally univalent in $\mathbb{D}.$ In recent years, Firoz Ali and Sanjit Pal \cite{Ali-Pal-2023, Ali-Pal--2023, Ali-Pal---2023, Ali-Pal-2024} have studied pre-Schwarzian norm for different classes of analytic functions. 
	Several important global univalence criteria for a locally univalent analytic function have been obtained using pre-Schwarzian norm. This makes a valid question if the pre-Schwarzian norm can be generalized to harmonic mappings. Positively, Kanas and Klimek-Sm\c{e}t \cite{Kanas-Smet-2014} introduced a definition of pre-Schwarzian derivative (which we denote by $\mathbf{P}_f$) for a locally univalent and sense preserving harmonic mapping $f = h+\overline{g}$ having dilatation $\omega = g'/h'= p^2$ for some analytic function $p$. The definition is given as follows:
	\begin{align}\label{p1-020}
		\mathbf{P}_f= \frac{2\partial(\log\lambda)}{\partial z}= \frac{h''}{h'}+\frac{2\overline{p}p'}{1+|p|^2},~~\text{where}~\lambda = |h'|+|g'|.
	\end{align}
	The pre-Schwarzian norm is defined by \eqref{P-141}.  In \cite{Ali-Pandit-2023}, Ali and Pandit have corrected some earlier flawed estimates of $\mathbf{P}_f$ for different classes of close-to-convex harmonic mappings having different analytic parts. 
	
	The pre-Schwarzian derivative given in  \eqref{p1-020} is valid for restricted  dilatation. In 2015, Hern{\'a}ndez and Mart{\'\i}n \cite{Hernandez-Martin-2015} defined the pre-Schwarzian derivative of a locally univalent harmonic mapping $f=h+\overline{g}$ by\\
	\begin{align}\label{p1-027}
		P_f = \left(\log J_f\right)_z = \frac{h''}{h'}-\frac{\overline{\omega}\omega'}{1-|\omega|^2},
	\end{align}
	where $J_f$ is the Jacobian and $\omega=g'/h'$ is the dilatation of $f$. Again, the pre-Schwarzian norm is defined same as \eqref{P-141}. In 2015, Hern{\'a}ndez and Mart{\'\i}n \cite{Hernandez-Martin-2015} proved the sharp estimate $\|P_f\|\le 5$ for convex harmonic mappings $f$. In 2016, Graf \cite{GRAF-2016} obtained sharp estimate of $\|P_f\|$ for a affine and linear invariant family of locally univalent harmonic mappings. In 2019, Liu and Ponnusamy \cite{Liu-Ponnusamy-2018} obtained the sharp estimates of the pre-Schwarzian norm $\|P_f\|$ for stable harmonic univalent functions and stable harmonic convex functions. In the same article \cite{Liu-Ponnusamy-2018}, authors have connected the pre-Schwarzian norm of harmonic mappings with pre-Schwarzian norm of its analytic part. In this direction, Ali and Pandit \cite{Ali-Pandit-2023} considered a class of close-to-convex harmonic mappings and obtained estimate of $\|P_f\|$ by introducing a different method of constructing extremal function. In this article, we obtain sharp estimate of pre-Schwarzian norm of certain close-to-convex harmonic mappings.\\

	\section{Main Results with Background}\label{Sect-2}
	Geometric properties of harmonic mappings are closely related with that of analytic functions. In particular, analytic parts of harmonic mappings play an important role to shape their geometric properties. It is well known that a sense preserving harmonic mapping $f=h+\overline{g}\in\mathcal{H}$ is univalent and close-to-convex if $h$ is univalent and convex (see \cite{Clunie-Small-1984}). In 2011, Bshouty and Lyzzaik \cite{Bshouty-Lyzzaik-2011} proved that a harmonic mapping $f=h+\overline{g}\in\mathcal{H}$ with dilatation $\omega=z$ is univalent and close-to-convex if 
	\begin{align*}
	h\in\mathcal{G}=\left\{\phi\in\mathcal{A}:{\rm Re\,}\left(1+\frac{z\phi''(z)}{\phi'(z)}\right)>-\frac{1}{2}\right\}
	\end{align*}
	   for $z\in\mathbb{D}$. We note that the class $\mathcal{G}$ can be obtained from the class $\mathcal{F}$ by taking $\alpha\rightarrow\infty.$ In this article, we consider a class $\mathcal{H}_c$, $-1/2\le c\le 0,$ of sense preserving harmonic mappings $f=h+\overline{g}\in\mathcal{H}$ in the unit disk $\mathbb{D}$ having dilatation $\omega=g'/h'$ and 
	   
	   \begin{align*}
	   h\in\mathcal{A}_c=	\left\{\phi\in\mathcal{A}:{\rm Re\,}\left(1+\frac{z\phi''(z)}{\phi'(z)}\right)>c\right\}.
	   \end{align*}  
In 2013, Bshouty et al. \cite{Bshouty-Joshi-Joshi-2013}   
prove that the sense preserving harmonic mapping $f=h+\overline{g}\in\mathcal{H}_c$ is univalent and close-to-convex if the dilatation $\omega(z)$
satisfies the condition $|\omega(z)|<\cos{(\pi c)}$ for $z\in\mathbb{D}$. In 2015, Ponnusamy and Kaliraj \cite{Ponnusamy-Kaliraj-2015} proved that  $f=h+\overline{g}\in\mathcal{H}_c$ is univalent and close-to-convex if the dilatation $\omega(z)$
satisfies the condition
\begin{align*}
{\rm Re\,}\left(1+\frac{\lambda z\omega'(z)}{1-\lambda\omega(z)}\right)>-\frac{1}{2},
\end{align*}
for all $\lambda$ such that $|\lambda|=1,~z\in\mathbb{D}$.

Now, we discuss the pre-Schwarzian norm estimate for harmonic mappings in the class $\mathcal{H}_c.$ 

	\begin{thm}\label{p807}
		Let $f=h+\overline{g}\in\mathcal{H}_c, -1/2\le c\le 0$ be a harmonic mapping of the form \eqref{P-151}. Then pre-Schwarzian norm $\|P_f\|\le 4(1-c)+1$. The estimate is sharp. 
	\end{thm}
	\begin{proof}
		As $f=h+\overline{g}\in\mathcal{H}_c$, from the definition, it follows that
		$$1+\frac{zh''(z)}{h'(z)}\prec\frac{1+(1-2c)z}{1-z},$$ which implies that,
		$$\frac{zh''(z)}{h'(z)}\prec\frac{2(1-c)z}{1-z}.$$ Thus, we get
		\begin{align}\label{p805}
			\left|\frac{zh''(z)}{h'(z)}\right|\le \frac{2|z|(1-c)}{1-|z|}.
		\end{align}
		This estimate is sharp for the analytic function $h(z)=\frac{(1-z)^{2c-1}-1}{1-2c}.$\\
		
		The pre-Schwarzian derivative $P_f$ of a harmonic mapping $f=h+\overline{g}\in\mathcal{H}_c$ is 
		$$
			P_f=\frac{h''}{h'}-\frac{\overline{\omega}\omega'}{1-|\omega|^2}.
		$$
		Thus from Schwarz Pick Lemma and the relation \eqref{p805}, the pre-Schwarzian norm $\|P_f\|$ is 
		\begin{align}\label{HM-100}
			\|P_f\|= & \sup_{z \in \mathbb{D}}\left|\frac{h''}{h'}-\frac{\overline{\omega}\omega'}{1-|\omega|^2}\right|(1-|z|^2)\\\nonumber
			\le & \sup_{z \in \mathbb{D}}\left(\left|\frac{h''}{h'}\right|+\left|\frac{\overline{\omega}\omega'}{1-|\omega|^2}\right|\right)(1-|z|^2)\\\nonumber
			\le & \sup_{z \in \mathbb{D}}\left(\left|\frac{2(1-c)}{1-|z|}\right|+\frac{1}{1-|z|^2}\right)(1-|z|^2)\\\nonumber
			=& \sup_{z \in \mathbb{D}}\left(2(1+|z|)(1-c)+1\right)\\\nonumber
			=& 4(1-c)+1.
		\end{align}
		
		To show that the estimate is sharp, we consider a harmonic mapping $f_d=h+\overline{g}$ with $h(z)=\frac{(1-z)^{2c-1}-1}{1-2c}$ and dilatation $\omega_d(z)=\frac{z-d}{1-dz},~d\in(0,1).$ It is clear that $f_d\in \mathcal{H}_c$ and so $\|P_{f_d}\|\le 4(1-c)+1.$ A few calculations show that 
		\begin{align}
			P_h(z)=\frac{2(1-c)}{1-z}\quad\text{and}\quad \frac{\overline{\omega_d(z)}\omega_d'(z)}{1-|\omega_d(z)|^2}=\frac{\overline{z}-d}{(1-dz)(1-|z|^2)}
		\end{align}
		from which it follows that 
		\begin{align}
			\|P_{f_d}\|=&\sup_{z \in \mathbb{D}}\left|\frac{h''(z)}{h'(z)}-\frac{\overline{\omega_d(z)}\omega_d'(z)}{1-|\omega_d(z)|^2}\right|(1-|z|^2)\\\nonumber
			=&\sup_{z \in \mathbb{D}}\left|\frac{2(1-c)}{1-z}-\frac{\overline{z}-d}{(1-dz)(1-|z|^2)}\right|(1-|z|^2)\\\nonumber
			\ge & L_d
		\end{align}
		where 
		\begin{align}
			L_d=&\sup_{z \in [0,1)}\left|\frac{2(1-c)}{1-z}-\frac{\overline{z}-d}{(1-dz)(1-|z|^2)}\right|(1-|z|^2)\\\nonumber\\\nonumber
			=&\sup_{r \in [0,1)}\left|2(1-c)(1+r)-\frac{r-d}{1-dr}\right|\\\nonumber
			=&\sup_{r \in [0,1)}M(r)\\\nonumber
		\end{align}
		with $M(r)=2(1-c)(1+r)-\frac{r-d}{1-dr}.$ We note that $M(0)=2(1-c)+d,~\phi(1^-)=4(1-c)-1.$ Next $M'(r)=0$ implies that $$r=\frac{1}{d}\pm\frac{1}{d}\sqrt{\frac{1-d^2}{2(1-c)}}.$$ We consider $\frac{1}{d}-\frac{1}{d}\sqrt{\frac{1-d^2}{2(1-c)}}=r_0$ as it lies in $(0,1).$ Since $r_0$ is stationary point and $$M(r_0)=2(1-c)+\frac{3-2c}{d}-\frac{4(1-c)}{d}\sqrt{\frac{1-d^2}{2(1-c)}}\rightarrow 4(1-c)+1$$ as $d$ tends to $1,$ it follows that $L_d=4(1-c)+1.$ Thus we get the relation $$4(1-c)+1=L_d\le \|P_{f_d}\|\le 4(1-c)+1$$ which implies that $\|P_{f_d}\|=4(1-c)+1.$
	\end{proof}
	It is important to note that the Theorem \ref{p807} provide pre-Schwarzian norm estimate of various classes of harmonic mappings depending on $c\in[-1/2,0].$ In this connection, we present the following result. 
\begin{cor}
	Let $f=h+\overline{g}\in\mathcal{H}$ be a harmonic mapping of the form \eqref{P-151}. If $h$ is convex then pre-Schwarzian norm $\|P_f\|\le 5$. The estimate is sharp. 
\end{cor}
	\begin{proof}
	From the hypothesis, it follows that $c=0.$ Thus the result directly comes from Theorem \ref{p807} upon putting $c=0$.
	\end{proof}

\begin{thm}
	A harmonic mapping $f=h+\overline{g}\in\mathcal{H}_c$ may not be harmonic Bloch.
\end{thm}	
\begin{proof}
	We proof the result by providing a harmonic mapping $f\in\mathcal{H}_c$ which is not harmonic Bloch. We consider the harmonic mapping $f=h+\overline{g}\in\mathcal{H}_c$ such that $$h(z)=\frac{(1-z)^{2c-1}}{1-2c}\quad\text{and}\quad \omega(z)=z$$
	Now 
	\begin{align*}
		\beta_f=&\sup_{z\in\mathbb{D}}(1-|z|^2)(|h'(z)|+|g'(z)|)\\
		=&\sup_{z\in\mathbb{D}}\left(\frac{(1-|z|^2)(1+|z|}{(1-z)^{2(1-c)}})\right)\\
		\le &\sup_{z\in(0,1)}\left(\frac{(1-z)(1+z)^2}{(1-z)^{2(1-c)}})\right)\\
		=&\sup_{z\in(0,1)}\left(\frac{(1+z)^2}{(1-z)^{2(1-c)-1}})\right)
	\end{align*}
	which is unbounded for $z\in(0,1).$ Thus $f$ is not harmonic Bloch.
\end{proof}

	Now onward, we discuss some growth and distortion theorems for harmonic mappings in the class $\mathcal{H}_c.$
	
	\begin{thm}\label{p870}
		Let $f=h+\overline{g}\in\mathcal{H}$ be a harmonic mapping of the form \eqref{P-151}. Then
		\begin{align*}
		\frac{1}{(1+|z|)^{2(1-c)}}\le |h'(z)|\le \frac{1}{(1-|z|)^{2(1-c)}}\quad\text{and}\quad 0\le |g'(z)|\le\frac{1}{1-|z|^{2(1-c)}}.
		\end{align*}
	\end{thm}
\begin{proof}
From the hypothesis, it follows that
$${\rm Re\,}\left(1+\frac{zh''(z)}{h'(z)}\right)> c$$ which implies that 
$$1+\frac{zh''(z)}{h'(z)}\prec\frac{1+(1-2c)z}{1-z},$$ that is, $$\frac{zh''(z)}{h'(z)}\prec z\frac{2(1-c)}{1-z}=\psi(z).$$
Next from a well known subordination result by Suffridge \cite{Suffridge-1970}, we get
\begin{align*}
h'(z)\prec &  e^{\int_{0}^{z}\frac{\psi(t)}{t}dT}\\\nonumber
  =& e^{\int_{0}^{z}\frac{2(1-c)}{1-t}dt}\\\nonumber
  = & \frac{1}{(1-z)^{2(1-c)}}.
\end{align*}
Thus
\begin{align}\label{p840}
	\frac{1}{(1+|z|)^{2(1-c)}}\le |h'(z)|\le \frac{1}{(1-|z|)^{2(1-c)}}.
	\end{align}
Now from the definition of dilatation $\omega=g'/h'$, it follows that 
\begin{align}\label{p850}
	\frac{\left||z|-\alpha\right|}{1-\alpha|z|}\le\left|\omega(z)\right|\le\frac{|z|+\alpha}{1+\alpha|z|},~~\alpha=|\omega(0)|.
\end{align}
Now applying \eqref{p840} and \eqref{p850} to the relation $g'=\omega h'$, we get
\begin{align*}
	\frac{\left||z|-\alpha\right|}{1-\alpha|z|}\frac{1}{(1+|z|)^{2(1-c)}}\le |g'(z)|\le \frac{|z|+\alpha}{1+\alpha|z|}\frac{1}{(1-|z|)^{2(1-c)}}.
\end{align*}
The left hand quantity $$\frac{\left||z|-\alpha\right|}{1-\alpha|z|}\frac{1}{(1+|z|)^{2(1-c)}}$$ has minimum value $0$ at $|z|=\alpha$ and the right hand quantity $$\frac{|z|+\alpha}{1+\alpha|z|}\frac{1}{(1-|z|)^{2(1-c)}}$$ has maximum value $\frac{1}{(1-|z|)^{2(1-c)}}$ as $\alpha\rightarrow 1^-.$ So we find that
$$0\le |g'(z)|\le\frac{1}{1-|z|^{2(1-c)}}.$$

To show that the estimates are sharp, we consider the harmonic mapping $f=h+\overline{g}$ with $h(z)=\{(1-z)^{2c-1}-1\}/(1-2c)$ and dilatation $\omega(z)=\frac{z+\alpha}{1+\alpha z},~\alpha\in(0,1).$ 
For $z\in(0,1)$ a few calculation show that
$h'(z)=\frac{1}{(1-|z|)^{2(1-c)}}$ and $h'(-z)=\frac{1}{(1+|z|)^{2(1-c)}}.$ Moreover, we get
$$g'(z)=\frac{z+\alpha}{1+\alpha z}\frac{1}{(1-z)^{2(1-c)}}$$, from which, for $z\in(0,1)$, it follows that 
$$g'(z)=\frac{z+\alpha}{1+\alpha z}\frac{1}{(1-z)^{2(1-c)}}\quad\text{and}\quad g'(-z)=\frac{-z+\alpha}{1-\alpha z}\frac{1}{(1-z)^{2(1-c)}}.$$
Now we  find that $g'(-z)=0$ at $z=\alpha$ and $g'(z)=\frac{1}{(1-z)^{2(1-c)}}$ as $\alpha\rightarrow 1^-$. This completes the proof.
\end{proof}

Before presenting our next result, we recall Gauss Hypergeometric function
\begin{align}\label{p900}
	{}_2F_1(a,b;c;z)=\sum\limits_{n=0}^\infty\frac{(a)_n(b)_n}{(c)_n}\frac{z^n}{n!},~~|z|<1,
\end{align}
	where $(q)_n$ is the Pochhammer symbol defined by
	$$(q)_0=1,~~~(q)_n=q(q+1)\dots (q+(n-1)),~~n\ge 1.$$

	\begin{thm}\label{log-000200}
		Let $f\in\mathcal{H}_c$ be a sense preserving harmonic mapping with dilatation $\omega=g'/h'.$ Then for $|z|=r$ the sharp inequalities
		\begin{align*}
			|g(z)|\le
			\begin{cases}
				\frac{1-(1-r)^{2c-1}}{\alpha(2c-1)}+(\alpha-\frac{1}{\alpha})\frac{{}_2F_1(1,2c-1;2c;\gamma)-(1-r)^{2c-1}{}_2F_1(1,2c-1;2c;\gamma(1-r))}{(1+\alpha)(2c-1)} ,~~\text{when}~~|\omega(0)|=\alpha\neq 0 \\
				 \frac{1-(1-r)^{2c-1}}{2c-1}-\frac{1-(1-r)^{2c-1}}{2c},\quad\quad\quad\quad\quad\quad\quad\quad\quad\quad\quad\quad\quad\quad\quad\quad\text{when}~~\omega(0)=0
			\end{cases}
		\end{align*}
		hold.
	\end{thm}
	\begin{proof}
	For $f=h+\overline{g}\in\mathcal{H}_c$ and dilatation $\omega=g'/h' :\mathbb{D}\rightarrow\mathbb{D},$ we have 
	$$|h'(z)|=\frac{1}{(1-r)^{2(1-c)}}\quad\text{and}\quad |\omega(z)|\le\frac{r+\alpha}{1+\alpha r},~~\alpha=|\omega(0)|.$$
	Thus for $\alpha\neq 0,$ we get
	\begin{align*}
	|g'(z)|\le  & \int_{0}^{r}|\omega(\zeta)|h'(z)|d|\zeta|\\
	\le & \int_{0}^{r}\frac{\alpha+|\zeta|}{1+\alpha|\zeta|}\frac{1}{(1-|\zeta|)^{2(1-c)}}d|\zeta|\\
	=&\frac{1-(1-r)^{2c-1}}{\alpha(2c-1)}+(\alpha-\frac{1}{\alpha})\frac{{}_2F_1(1,2c-1;2c;\gamma)-(1-r)^{2c-1}{}_2F_1(1,2c-1;2c;\gamma(1-r))}{(1+\alpha)(2c-1)},
	\end{align*}
	where $\gamma=\frac{\alpha}{1+\alpha}$ and the quantity ${}_2F_1(a,b;c;z)$ is defined by \eqref{p900}.\\
	
	 For $\alpha=0$, it is well known that $|\omega(z)|\le|z|$ and so
	\begin{align*}
	|g'(z)|\le  & \int_{0}^{r}|\omega(\zeta)|h'(z)|d|\zeta|\\
	\le & \int_{0}^{r}|\zeta|\frac{1}{(1-|\zeta|)^{2(1-c)}}d|\zeta|\\
	=&\frac{1-(1-r)^{2c-1}}{2c-1}-\frac{1-(1-r)^{2c-1}}{2c}.
	\end{align*}
	Sharpness hold for the harmonic mapping $f=h+\overline{g}\in\mathcal{H}_c$ such that
	$$h(z)=\frac{(1-z)^{2c-1}-1}{1-2c}$$ and dilatation $\omega(z)=\frac{z+\alpha}{1+\alpha z},~~\alpha\in[0,1)$.
	
	\end{proof}

	%%%%%%%%%%%%%%%%%%%%%%%%%%%%%%%%%%%%%%%%%%%%%%%%%%%%%%%%%%%%%%%%%%%%%%%%%%%%%%%%%%%%%%%%%%%%%%%%%%%%%%%%%%%%%%%%%%%%%%%%%%%%%%%%%%%%%%%%%%%%%%

	\section*{Declarations:}

	\noindent\textbf{Data availability:}
	Data sharing not applicable to this article as no data sets were generated or analyzed during the current study.\\

	\noindent\textbf{Conflict of interest:} The author declares that he has no conflict of interest.

	%Yamashita, Shinji. "The Pick version of the Schwarz lemma and comparison of the Poincaré densities." Annales Fennici Mathematici 19, no. 2 (1994): 291-322.
	
\end{document}